\documentclass[11pt]{amsart}

\usepackage{amssymb,enumerate}

\theoremstyle{plain}
\newtheorem{theorem}{Theorem}[section]
\newtheorem{proposition}[theorem]{Proposition}
\newtheorem{lemma}[theorem]{Lemma}
\newtheorem{sublemma}{}[theorem]
\newtheorem{subsublemma}{}[sublemma]

\newenvironment{subproof}
{\begin{proof}}
{
\end{proof}}

\newcommand{\dash}{\nobreakdash-\hspace{0pt}}
\newcommand{\ifc}{internally $4$\dash connected}
\newcommand{\ftv}{$(4,3)$\dash violator}
\newcommand{\ba}{\backslash}
\newcommand{\cl}{\operatorname{cl}}
\newcommand{\symdif}{\triangle}

\title
[A splitter theorem for binary matroids II]
{Towards a splitter theorem for internally $4$-connected
binary matroids II}

\author[Chun]{Carolyn Chun}
\address{School of Mathematics, Statistics, and
Operations Research,
Victoria University,
Wellington,
New Zealand}
\email{carolyn.chun@vuw.ac.nz}

\author[Mayhew]{Dillon Mayhew}
\address{School of Mathematics, Statistics, and
Operations Research,
Victoria University,
Wellington,
New Zealand}
\email{dillon.mayhew@vuw.ac.nz}

\author[Oxley]{James Oxley}
\address{Department of Mathematics,
Louisiana State University,
Baton Rouge,
LA,
USA}
\email{ oxley@math.lsu.edu }


\date{\today}

\begin{document}

\begin{abstract}
Let $M$ and $N$ be internally $4$\dash connected binary matroids
such that $M$ has a proper $N$\dash minor, and
$|E(N)|\geq 7$.
As part of our project to develop a splitter theorem for
internally $4$\dash connected binary matroids, we prove the
following result:
if $M\ba e$ has no $N$\dash minor whenever $e$ is in a triangle
of $M$, and $M/e$ has no $N$\dash minor whenever $e$ is in
a triad of $M$, then $M$ has a minor, $M'$, such that
$M'$ is internally $4$\dash connected with an $N$\dash minor, and
$1\leq |E(M)|-|E(M')|\leq 2$. 
\end{abstract}

\maketitle

\section{Introduction}

It would be useful for structural matroid theory if we could
make the following statement:
there exists an integer, $k$, such that whenever $M$ and $N$ are
\ifc\ binary matroids and $M$ has a proper $N$\dash minor, then
$M$ has an \ifc\ minor, $M'$, such that $M'$ has an $N$\dash minor,
and $1\leq |E(M)|-|E(M')|\leq k$.
However this statement is false; no such $k$ exists.
To see this, we let $M$ be the cycle matroid
of a quartic planar ladder on $n$ vertices, and we let $N$
be the cycle matroid of the cubic planar ladder on the
same number of vertices.
Then $M$ and $N$ are \ifc, and $M$ has a proper minor
isomorphic to $N$.
Moreover, $|E(M)|=2n$, and $|E(N)|=3n/2$.
However, the only proper minor of $M$ that is \ifc\ with an
$N$\dash minor is isomorphic to $N$.

In light of this example, we concentrate on a different goal.
To aid brevity, let us introduce some notation.
Say that $\mathcal{S}$ is the set of all ordered pairs,
$(M,N)$ where $M$ and $N$ are \ifc\ binary matroids, and
$M$ has a proper $N$\dash minor.
We will let $\mathcal{S}_{k}$ be the subset of
$\mathcal{S}$ for which there is an \ifc\ minor,
$M'$, of $M$ that has an $N$\dash minor and
satisfies $1\leq |E(M)|-|E(M')|\leq k$.
The discussion in the previous paragraph shows that we cannot
find a $k$ so that
$\mathcal{S}\subseteq \mathcal{S}_{k}$.
Instead, we want to show that, for any $(M,N)\in\mathcal{S}$,
either $(M,N)\in \mathcal{S}_{k}$, for some small value of
$k$, or there is some easily described operation
we can perform on $M$ to produce an \ifc\ minor that has
an $N$\dash minor.
To this end, we are trying to identify as many pairs as
possible that belong to $\mathcal{S}_{k}$, for small values of
$k$.
For example, our first step \cite{CMO12} was to show that
if $M$ is $4$\dash connected, then $(M,N)$ is in $\mathcal{S}_{2}$.
In fact, in almost every case, $(M,N)$ belongs to
$\mathcal{S}_{1}$.

\begin{theorem}
\label{splitterI}
Let $M$ and $N$ be binary matroids such that
$M$ has a proper $N$\dash minor, and $|E(N)|\geq 7$.
If $M$ is $4$\dash connected and $N$ is \ifc, then $M$ has
an \ifc\ minor $M'$ with an $N$\dash minor such that
$1\leq |E(M)|-|E(M')|\leq 2$.
Moreover, unless $M$ is isomorphic to a specific $16$-element
self-dual matroid, such an $M'$ exists with $|E(M)| - |E(M')| = 1$.
\end{theorem}

An \ifc\ binary matroid is $4$\dash connected if and only
if it has no triangles and triads.
Therefore we have shown that if $M$ has no triangles or triads
(and $|E(N)|\geq 7$), then $(M,N)\in \mathcal{S}_{2}$.
Hence we now assume that $M$ does contain a
triangle or triad.
In this chapter of the series, we consider the case that all
triangles and triads of $M$ must be contained in the ground set
of every $N$\dash minor.
In other words, deleting an element from a triangle of $M$,
or contracting an element from a triad, destroys all
$N$\dash minors.
We show that under these circumstances,
$(M,N)$ is in $\mathcal{S}_{2}$.

\begin{theorem}
\label{main}
Let $M$ and $N$ be \ifc\ binary matroids, such that
$|E(N)| \geq 7$, and $N$ is isomorphic to a proper minor of $M$.
Assume that if $T$ is a triangle of $M$ and $e\in T$, then
$M\ba e$ does not have an $N$\dash minor.
Dually, assume that if $T$ is a triad and $e\in T$, then
$M/e$ does not have an $N$\dash minor.
Then $M$ has an \ifc\ minor, $M'$, such that
$M'$ has an $N$\dash minor, and $1\leq |E(M)|-|E(M')| \leq 2$.
\end{theorem}

With this result in hand, in the next chapter
\cite{SplitIII}
we will be able
to assume that (up to duality) $M$ has a triangle $T$ and an
element $e\in T$ such that $M\ba e$ has an $N$\dash minor.

We note that Theorem~\ref{main} is not strictly a
strengthening of Theorem~\ref{splitterI} as, in the earlier
theorem, we completely characterized when
$(M,N)$ was in $\mathcal{S}_{2}-\mathcal{S}_{1}$.
We make no attempt to obtain the corresponding characterization
in Theorem~\ref{main}, as we believe that
$\mathcal{S}_{2}-\mathcal{S}_{1}$ will contain many more pairs when
we relax the constraint that $M$ is $4$\dash connected.
For example, let $N$ be obtained from a binary projective geometry
by performing a $\Delta\text{-}Y$ exchange on a triangle $T$.
Let $T'$ be a triangle that is disjoint from $T$.
We obtain $M$ from $N$ by coextending by the element $x$
so that it is in a triad with two elements from $T'$, and then
extending by $y$ so that it is in a circuit with $x$
and two elements from $T$.
It is not difficult to confirm that the hypotheses of
Theorem~\ref{main} hold, but $M$ has no \ifc\ single-element
deletion or contraction with an $N$\dash minor.
Clearly this technique could be applied to create even more
diverse examples.

\section{Preliminaries}

We assume familiarity with standard matroid
notions and notations, as presented in \cite{Oxl11}.
We make frequent, and sometimes implicit, use of the following
well-known facts.
If $M$ is $n$\dash connected, and $|E(M)| \geq 2(n-1)$, then
$M$ has no circuit or cocircuit with fewer than
$n$ elements \cite[Proposition~8.2.1]{Oxl11}.
In a binary matroid, a circuit and a cocircuit must meet in a
set of even cardinality \cite[Theorem~9.1.2(ii)]{Oxl11}.
Morever, the symmetric difference, $C \symdif C'$,
of two circuits in a binary matroid is a
disjoint union of circuits \cite[Theorem~9.1.2(iv)]{Oxl11}.

We use `by orthogonality' as shorthand for the statement
`by the fact that a circuit and a cocircuit cannot
intersect in a set of cardinality one'
\cite[Proposition~2.1.11]{Oxl11}.
A \emph{triangle} is a $3$\dash element circuit, and
a \emph{triad} is a $3$\dash element cocircuit.
We use $\lambda_{M}$ or $\lambda$ to denote the connectivity function of
the matroid $M$.
If $M$ and $N$ are matroids,
an \emph{$N$\dash minor} of $M$ is
a minor of $M$ that is isomorphic to $N$.

Let $M$ be a matroid.
A subset $S$ of $E(M)$ is a \emph{fan} in $M$ if $|S|\geq 3$ and
there is an ordering $(s_{1},s_{2},\ldots,s_{n})$ of $S$ such that
\[\{s_{1},s_{2},s_{3}\},\{s_{2},s_{3},s_{4}\},\ldots,
\{s_{n-2},s_{n-1},s_{n}\}\]
is an alternating sequence of triangles and triads.
We call $(s_{1},s_{2},\ldots,s_{n})$ a \emph{fan ordering} of $S$.
Sometimes we blur the distinction between a fan and
an ordering of that fan.
Most of the fans we encounter have four or five elements.
We adopt the following convention: if
$(s_{1},s_{2},s_{3}, s_{4})$ is a fan ordering of a
$4$\dash element fan, then $\{s_{1},s_{2},s_{3}\}$ is a triangle.
We call such a fan ordering a \emph{$4$\dash fan}. 
We distinguish between the two different types of
$5$\dash element fan by using \emph{$5$\dash fan} to refer to
a $5$\dash element fan containing two triangles, and
using \emph{$5$\dash cofan} to refer to a $5$\dash element
fan containing two triads.

The next proposition is proved by induction on $n$, using the
fact that $s_{n}$ is contained in either the closure or
the coclosure of $\{s_{1},\ldots, s_{n-1}\}$.

\begin{proposition}
\label{fans3sep}
Let $(s_{1},\ldots, s_{n})$ be a fan ordering in a matroid
$M$.
Then
\[
\lambda_{M}(\{s_{1},\ldots, s_{n}\})\leq 2.
\]
\end{proposition}

\begin{lemma}
\label{minorsof45fans}
Let $M$ be a binary matroid that has an \ifc\ minor, $N$,
satisfying $|E(N)|\geq 8$.
If $(s_{1},s_{2},s_{3},s_{4})$ is a $4$\dash fan of $M$,
then $M\ba s_{1}$ or $M/s_{4}$ has an $N$\dash minor.
If $(s_{1},s_{2},s_{3},s_{4},s_{5})$ is a
$5$\dash fan in $M$, then either $M\ba s_{1}\ba s_{5}$ has an
$N$\dash minor, or both $M\ba s_{1}/s_{2}$ and $M/s_{4}\ba s_{5}$
have N\dash minors.
In particular, both $M\ba s_{1}$ and	$M\ba s_{5}$ have $N$\dash minors.
\end{lemma}

\begin{proof}
Let $(s_{1},s_{2},s_{3},s_{4})$ be a $4$\dash fan.
Since $\{s_{1},s_{2},s_{3},s_{4}\}$ contains	a circuit and a cocircuit,
$\lambda_{N_{0}}(\{s_{1},s_{2},s_{3},s_{4}\})	\leq 2$	for any minor,
$N_{0}$, of $E(M)$ that contains	$\{s_{1},s_{2},s_{3},s_{4}\}$ in
its ground set.
As $N$ is \ifc\ and $|E(N)|\geq 8$, we deduce that $N$ is obtained from
$M$ by removing at least one element of $\{s_{1},s_{2},s_{3},s_{4}\}$.
Let $x$ be an element in $\{s_{1},s_{2},s_{3},s_{4}\}-E(N)$.
If $M\ba x$ has an $N$\dash minor, then either $x=s_{1}$, as desired;
or $\{s_{2},s_{3},s_{4}\}-x$ is a $2$\dash cocircuit in $M\ba x$.
In the latter case, as $N$ is \ifc, either $x \in \{s_{2},s_{3}\}$, and
$M/s_{4}$ has an $N$\dash minor, as desired; or $x = s_{4}$, and
$M/s_{2}$ has an $N$\dash minor.
But $\{s_{1},s_{3}\}$ is a $2$\dash circuit of the last matroid, so
$M\ba s_{1}$ has an $N$\dash minor, and the lemma holds.
We may now suppose that deleting any element of
$\{s_{1},s_{2},s_{3},s_{4}\}$ from $M$ yields a matroid with no
$N$\dash minor.
Then $N$ is a minor of $M/x$ for some $x \in\{s_{1},s_{2},s_{3},s_{4}\}$. But $x$ is not in $\{s_{1},s_{2},s_{3}\}$, or else
$\{s_{1},s_{2},s_{3}\}-x$ is a $2$\dash circuit in $M/x$, and we
may delete one of its elements while keeping an $N$\dash minor.
Thus $x = s_{4}$, and the lemma holds.

Next we assume that
$(s_{1},s_{2},s_{3},s_{4},s_{5})$ is a $5$\dash fan in $M$.
First we show that $M\ba s_{1}/s_{2}$ has an $N$\dash minor
if and only if $M/s_{4}\ba s_{5}$ has an $N$\dash minor.
As $\{s_{1},s_{3}\}$ is a $2$\dash circuit of $M/s_{2}$,
it follows that if $M\ba s_{1}/s_{2}$ has an $N$\dash minor,
so does $M\ba s_{3}$.
As $\{s_{4},s_{5}\}$ is a $2$\dash cocircuit of the last matroid,
this implies that $M/s_{4}$ has an $N$\dash minor.
Hence so does $M/s_{4}\ba s_{5}$.
Thus $M/s_{4}\ba s_{5}$ has an $N$\dash minor if
$M\ba s_{1}/s_{2}$ does.
The converse statement yields
to a symmetrical argument.

Now  $(s_{1}, s_{2},s_{3},s_{4})$ is a $4$\dash fan of $M$.
By applying the first statement of the lemma, we see that
$M\ba s_{1}$ or $M/s_{4}$ has an $N$\dash minor.
In the latter case, $M/s_{4}\ba s_{5}$ has an
$N$\dash minor, and we are done.
Therefore we assume that $M\ba s_{1}$ has an $N$\dash minor.
There is a cocircuit of $M\ba s_{1}$ that contains $s_{2}$
and is contained in $\{s_{2},s_{3},s_{4}\}$.
If this cocircuit is not a triad, then $M\ba s_{1}/s_{2}$ has an
$N$\dash minor, and we are done.
Therefore we assume that
$(s_{5},s_{4},s_{3},s_{2})$ is a $4$\dash fan of $M\ba s_{1}$.
We apply the first statement of the lemma, and deduce that
either $M\ba s_{1}/s_{2}$ or $M\ba s_{1}\ba s_{5}$ has
an $N$\dash minor.
In either case the proof is complete.
\end{proof}

A \emph{quad} is a $4$\dash element circuit-cocircuit.
It is clear that if $Q$ is a quad, then
$\lambda(Q)\leq 2$.
The next result is easy to verify.

\begin{proposition}
\label{quad4fan}
Let $(X,Y)$ be a $3$\dash separation of a $3$\dash connected
binary matroid with $|X|=4$.
Then $X$ is a quad or a $4$\dash fan.
\end{proposition}

The next result is Lemma~2.2 in \cite{CMO12}.

\begin{lemma}
\label{quadiso}
Let $Q$ be a quad in a binary matroid $M$.
If $x$ and $y$ are in $Q$, then $M\ba x$ and $M\ba y$
are isomorphic.
\end{lemma}

A matroid is \emph{$(4,k)$\dash connected} if
it is $3$\dash connected, and, whenever $(X,Y)$ is a
$3$\dash separation, either $|X|\leq k$ or $|Y|\leq k$.
A matroid is \emph{\ifc} precisely when it is
$(4,3)$\dash connected.
If a matroid is $3$\dash connected, but not
$(4,k)$\dash connected, then it contains a $3$\dash separation,
$(X,Y)$, such that $|X|,|Y|> k$.
We will call such a $3$\dash separation a
\emph{$(4,k)$\dash violator}.

For $n\geq 3$, we let $G_{n+2}$ denote the
\emph{biwheel} graph with $n+2$ vertices.
Thus $G_{n+2}$ consists of a cycle
$v_{1},v_{2},\ldots, v_{n}$, and two additional
vertices, $u$ and $v$, each of which is adjacent to every vertex in
$\{v_{1},v_{2},\ldots, v_{n}\}$.
The planar dual of a biwheel is a \emph{cubic planar ladder}.
We construct $G_{n+2}^{+}$ by adding an edge between
$u$ and $v$.
It is easy to see that $M(G_{n+2}^{+})$ is
represented over $\mathrm{GF}(2)$ by
the following matrix
\[
\left[
\begin{array}{c|cc}
&\mathbf{1}&\mathbf{0}\\[-1ex]
I_{n+1}&&\\[-1ex]
&I_{n}&A_{n}\\
\end{array}
\right]
\]
where $A_{n}$ is the $n\times n$ matrix
\[
\begin{bmatrix}
1&0&0&\cdots&1\\
1&1&0&\cdots&0\\
0&1&1&\cdots&0\\
\vdots&\vdots&\vdots&\ddots&\vdots\\
0&0&0&\cdots&1
\end{bmatrix}
\]
and $\mathbf{1}$ and $\mathbf{0}$ are
$1\times n$ vectors with all entries equal to
$1$ or $0$ respectively.
Thus $M(G_{n+2}^{+})$ is precisely equal to
the matroid $D_{n}$, as defined by Zhou~\cite{zhou},
and the element $f_{1}$ of $E(D_{n})$ is the
edge $uv$.

For $n\geq 2$ let $\Delta_{n+1}$ be the
rank\dash$(n+1)$ binary matroid
represented by the following matrix.
\[
\left[
\begin{array}{c|cc}
&\mathbf{1}&e_{n}\\[-1ex]
I_{n+1}&&\\[-1ex]
&I_{n}&A_{n}\\
\end{array}
\right]
\]
In this case, $e_{n}$ is the
standard basis vector
with a one in position $n$.
Then $\Delta_{n+1}$ is a \emph{triangular M\"{o}bius matroid}
(see \cite{MRW10}).
In \cite{zhou}, the notation $D^{n}$ is used for the
matroid $\Delta_{n+1}$, and $f_{1}$ denotes the element
represented by the first column in the matrix.
We use $z$ to denote the same element.
We observe that $\Delta_{n+1}\ba z$ is the bond matroid of a
\emph{M\"{o}bius cubic ladder}.

The next result is a consequence of a theorem
due to Zhou~\cite{zhou}.

\begin{theorem}
\label{zhou}
Let $M$ and $N$ be \ifc\ binary matroids such that
$N$ is a proper minor of $M$ satisfying $|E(N)|\geq 7$.
Then either
\begin{enumerate}[(i)]
\item $M\ba e$ or $M/e$ is $(4,4)$\dash connected with
an $N$\dash minor, for some element $e\in E(M)$, or
\item $M$ or $M^{*}$ is isomorphic to
either $M(G_{n+2})$, $M(G_{n+2}^{+})$,
$\Delta_{n+1}$, or $\Delta_{n+1}\ba z$, for
some $n\geq 4$.
\end{enumerate}
\end{theorem}

Note that the theorem in \cite{zhou} is stated
with the weaker hypothesis that
$|E(N)|\geq 10$.
However, Zhou explains that by using results from
\cite{GZ06} and \cite{Zho04} and performing a
relatively simple case-analysis, we can strengthen the
theorem so that it holds under the condition that $|E(N)|\geq 7$.

\section{Proof of the main theorem}

In this section we prove Theorem~\ref{main}.
Throughout the section, we assume that the theorem is
false.
This means that there exist \ifc\ binary matroids,
$\bar{M}$ and $\bar{N}$, with the following properties:
\begin{enumerate}[(i)]
\item $\bar{M}$ has a proper $\bar{N}$\dash minor,
\item if $e$ is in a triangle of $\bar{M}$, then
$\bar{M}\ba e$ has no $\bar{N}$\dash minor,
\item if $e$ is in a triad of $\bar{M}$, then
$\bar{M}/ e$ has no $\bar{N}$\dash minor,
\item there is no \ifc\ minor, $M'$, of $\bar{M}$ such that $M'$
has an $\bar{N}$-minor and $1\leq |E(\bar{M})|-|E(M')| \leq 2$, and
\item $|E(\bar{N})|\geq 7$.
\end{enumerate}

Note that $(\bar{M}^{*},\bar{N}^{*})$ also provides a counterexample
to Theorem~\ref{main}.
We start by showing that we can assume
$|E(\bar{N})|\geq 8$.
If $|E(\bar{N})|=7$, then 
$\bar{N}$ is isomorphic to $F_{7}$ or $F_{7}^{*}$.
Then $\bar{M}$ is non-regular, and contains one of the five
\ifc\ non-regular matroids $N_{10}$,
$\widetilde{K}_{5}$,
$\widetilde{K}_{5}^{*}$,
$T_{12}\ba e$,
or $T_{12}/e$
as a minor
\cite[Corollary~1.2]{Zho04}.
But $N_{10}$ contains an element in a triangle
whose deletion is non-regular, so $M$ is
not isomorphic to $N_{10}$.
The same statement applies to $\widetilde{K}_{5}$
and $T_{12}/e$, so $\bar{M}$ is not isomorphic to
these matroids, or their duals,
$\widetilde{K}_{5}^{*}$
and $T_{12}\ba e$.
Thus $\bar{M}$ has a proper \ifc\ minor, $N'$, isomorphic to one
of the five matroids listed above.
Therefore we can relabel $N'$ as $\bar{N}$.
As each of the five matroids has more than seven elements,
we are justified in assuming that $|E(\bar{N})|\geq 8$.
As $(\bar{M},\bar{N})$ provides a counterexample
to Theorem~\ref{main}, it follows that $|E(\bar{M})|\geq 11$.

\begin{lemma}
\label{trianglespersist}
Let $(M,N)$ be $(\bar{M},\bar{N})$ or
$(\bar{M}^{*},\bar{N}^{*})$.
Let $N_{0}$ be an arbitrary $N$\dash minor of $M$.
If $T$ is a triangle or a triad of $M$, then
$T\subseteq E(N_{0})$.
\end{lemma}

\begin{proof}
By duality, we can assume that $\{e,f,g\}$ is a triangle
of $M$, and that $e\notin E(N_{0})$.
Since $M\ba e$ has no minor isomorphic to $N$,
it follows that $N_{0}$ is a minor of $M/e$.
As $\{f,g\}$ is a $2$\dash circuit in $M/e$,
it follows that
$N_{0}$ is a minor of either $M/e\ba f$ or
$M/e\ba g$, and hence of $M\ba f$ or $M\ba g$.
But neither of these matroids has an $N$\dash minor,
so we have a contradiction.
\end{proof}

\begin{lemma}
\label{candeletee}
There is an element $e\in E(\bar{M})$ such that
either $\bar{M}\ba e$ or $\bar{M}/ e$ is
$(4,4)$\dash connected with an $\bar{N}$\dash minor.
\end{lemma}

\begin{proof}
If the lemma fails, then by Theorem~\ref{zhou}, either
$\bar{M}$ or its dual is isomorphic to one of
$M(G_{n+2})$, $M(G_{n+2}^{+})$,
$\Delta_{n+1}$, or $\Delta_{n+1}\ba z$, for
some $n\geq 4$.
In this case it is easy to verify that every element of $E(\bar{M})$
is contained in a triangle or a triad.
Therefore Lemma~\ref{trianglespersist} implies that
$E(\bar{M})=E(\bar{N})$, contradicting the fact that $\bar{N}$ is a
proper minor of $\bar{M}$.
\end{proof}

If $e$ is an element such that $\bar{M}\ba e$ is $(4,4)$\dash connected
with an $\bar{N}$\dash minor, then $\bar{M}\ba e$
has a quad or a $4$\dash fan, for otherwise it follows from
Proposition~\ref{quad4fan} that $\bar{M}\ba e$ is \ifc,
contradicting the fact that $\bar{M}$ is a counterexample
to Theorem~\ref{main}.
We will make frequent use of the following fact.

\begin{proposition}
\label{zhou2.15}
Let $(M,N)$ be either $(\bar{M},\bar{N})$ or $(\bar{M}^{*},\bar{N}^{*})$.
If $M\ba e$ is $3$\dash connected and has an $N$\dash minor,
and $(X,Y)$ is a $3$\dash separation of $M\ba e$ such that
$|Y|=5$, then $Y$ is a $5$\dash cofan of $M\ba e$.
\end{proposition}

\begin{proof}
If $Y$ is not a fan, then $Y$ contains a quad
(see \cite[Lemma~2.14]{zhou}).
As in the proof of \cite[Lemma~2.15]{zhou}, we can show that
in $M$, there is either a triangle or a triad of $M$ that is contained
in $Y$ and which contains two elements from the quad.
In the first case, the triangle contains an element we can delete
to keep an $N$\dash minor.
In the second case, the triad contains an element we can contract
and keep an $N$\dash minor.
In either case, we have a contradiction to
Lemma~\ref{trianglespersist}.
Therefore $Y$ is a $5$\dash element fan.
If $Y$ is a $5$\dash fan, then by
Lemma~\ref{minorsof45fans}, we can delete an element
from a triangle in $M\ba e$ and preserve an
$N$\dash minor.
This contradicts Lemma~\ref{trianglespersist}, so
$Y$ is a $5$\dash cofan.
\end{proof}

At this point, we give a quick summary of the lemmas that follow.
Lemma~\ref{fanlemma} considers the matroid produced by contracting
the last element of a $4$\dash fan in $M\ba e$.
Lemma~\ref{quad44con} deals with deleting an element from a quad
in $M\ba e$.
In Lemma~\ref{nodeletequads} we show that whenever we delete
such an element,
we destroy all $N$\dash minors.
We exploit this information in Lemma~\ref{no4fans}, and show that
$M\ba e$ has no $4$\dash fans.
The only case left to consider is one in which we contract
an element from a quad in $M\ba e$.
This case is covered in Lemma~\ref{conquad}.
After this lemma, there is only a small amount
of work to be done before we obtain a final contradiction
and complete the proof.

\begin{lemma}
\label{fanlemma}
Let $(M,N)$ be either $(\bar{M},\bar{N})$ or $(\bar{M}^{*},\bar{N}^{*})$.
Assume that $e$ is an element of $M$ such that
$M\ba e$ is $(4,4)$\dash connected with an $N$\dash minor, and that
$(a,b,c,d)$ is a $4$\dash fan of $M\ba e$.
Then $M\backslash e/d$ is $3$-connected with an $N$\dash minor,
and $M/d$ is $(4,4)$\dash connected.
Moreover, if $(X,Y)$ is a \ftv\ of $M/d$ such that
$|X\cap\{a,b,c\}|\geq 2$, then $Y$ is a quad of $M/d$, and
$Y\cap\{a,b,c,e\}=\{e\}$.
\end{lemma}

\begin{proof}
From Proposition~\ref{fans3sep} and the fact that $M$
is \ifc\ with at least eleven elements, it follows that
$(a,b,c,d)$ is not a $4$\dash fan of $M$.
Therefore $\{b,c,d,e\}$ is a cocircuit in $M$.

\begin{sublemma}
\label{Mcond3con}
$M\ba e/ d$ and $M/d$ are $3$\dash connected.
\end{sublemma}

\begin{subproof}
We start by showing that $M\ba e /d$ is $3$\dash connected.
Since $M\ba e$ is $(4,4)$\dash connected and
$|E(M\ba e)|\geq 10$, it follows that
$\{a,b,c,d\}$ is not properly contained in a fan of
$M\ba e$.
As $b$ and $c$ are contained in a triangle and a triad in
$M\ba e$, it follows that deleting or contracting either
of these elements from $M\ba e$ produces a matroid
that is not $3$\dash connected.
Now \cite[Lemma~8.8.6]{Oxl11} implies that
either $M\ba e\ba d$ or $M\ba e/d$ is $3$\dash connected.
The former matroid contains a $2$\dash cocircuit, so
$M\ba e/d$ is $3$\dash connected.

If $M/d$ is not $3$\dash connected, then it follows easily
(see the dual of \cite[Lemma~2.6]{Oxl81})
that $\{e,d\}$ is contained in a triangle of $M$.
However, $N$ is a minor of $M\ba e$, so we have a contradiction
to Lemma~\ref{trianglespersist}.
\end{subproof}

\begin{sublemma}
\label{Mdelecond}
$M\ba e/d$, and hence $M/d$, has an $N$\dash minor.
\end{sublemma}

\begin{subproof}
Let $N_{0}$ be an $N$\dash minor of $M$.
Then $\{a,b,c\}\subseteq E(N_{0})$, by Lemma~\ref{trianglespersist}.
As $(a,b,c,d)$ is a $4$\dash fan of $M\ba e$,
it now follows by Lemma~\ref{minorsof45fans} that
$M\ba e /d$ has an $N$\dash minor.
\end{subproof}

\begin{sublemma}
\label{3sepsMcond}
Let $(X,Y)$ be a \ftv\ of $M/d$, and assume that
$|X\cap\{a,b,c\}|\geq 2$.
Then $|Y|= 4$, and $e\in Y$.
Morever, $Y\cap\{a,b,c\}=\emptyset$.
\end{sublemma}

\begin{subproof}
Assume that the result fails.

\begin{subsublemma}
\label{Ydisbig}
$|Y|\geq 5$.
\end{subsublemma}

\begin{subproof}
Assume otherwise.
Then $|Y|=4$.
Assume that $e\in X$.
If $\{b,c\}\subseteq X$, then
$d\in\cl_{M}^{*}(X)$, as $\{b,c,d,e\}$ is a cocircuit.
It follows from~\cite[Corollary~8.2.6(iii)]{Oxl11} that
$\lambda_{M}(X\cup d)=\lambda_{M/d}(X)$, and therefore
$(X\cup d,Y)$ is a \ftv\ of $M$, an impossibility.
Hence either $b$ or $c$ is contained in $Y$.

Proposition~\ref{quad4fan} implies that $Y$
is either a quad or a $4$\dash fan of $M/d$.
As $\{a,b,c\}$ is a triangle of $M/d$ that meets $Y$
in a single element, $Y$ is not a cocircuit, and hence not
a quad of $M/d$.
Thus $Y=\{y_{1},y_{2},y_{3},y_{4}\}$, where
$(y_{1},y_{2},y_{3},y_{4})$ is a $4$\dash fan of $M/d$.
Since the triangle $\{a,b,c\}$ cannot meet the triad
$\{y_{2},y_{3},y_{4}\}$ in a single element, it
follows that $y_{1}$ is equal to $b$ or $c$.

Let $N_{0}$ be an $N$\dash minor of $M/d$.
Since $\{y_{2},y_{3},y_{4}\}$ is a triad of $M$, it follows from
Lemma~\ref{trianglespersist} that
$\{y_{2},y_{3},y_{4}\}\subseteq E(N_{0})$.
But $\{a,b,c\}$ is a triangle of $M$, so
$\{a,b,c\}\subseteq E(N_{0})$.
Therefore $\{y_{1},y_{2},y_{3},y_{4}\}\subseteq E(N_{0})$,
and this contradicts Lemma~\ref{minorsof45fans}.
Therefore we conclude that $e\in Y$.

Since \ref{3sepsMcond} fails, yet $|Y|=4$ and $e\in Y$,
we deduce that $Y$ contains exactly one element
of the triangle $\{a,b,c\}$.
Thus $Y$ is not a quad of $M/d$, so $Y$ is a $4$\dash fan,
$(y_{1},y_{2},y_{3},y_{4})$, of $M/d$.
Since $\{a,b,c\}$ is a triangle of $M/d$, and $\{y_{2},y_{3},y_{4}\}$
is a triad, orthogonality requires that the single
element in $Y\cap\{a,b,c\}$ is $y_{1}$.
Therefore $e$ is contained in the triad $\{y_{2},y_{3},y_{4}\}$.
But this means that $M\ba e$ contains a $2$\dash cocircuit,
a contradiction as it is $3$\dash connected.
\end{subproof}

Let $T=\{a,b,c\}$.
As $Y$ contains at most one element of $T$, it follows
from \ref{Ydisbig} that $|Y-T|\geq 4$.
Furthermore, $X$ spans $T$.
The next fact follows from these observations and from
\ref{Mcond3con}.

\begin{subsublemma}
\label{XdcupT}
$(X\cup T,Y-T)$ is a $3$\dash separation in $M/d$.
\end{subsublemma}

\begin{subsublemma}
\label{einYd}
$e\in Y$.
\end{subsublemma}

\begin{subproof}
Assume that $e\in X$.
Then~\ref{XdcupT} and the cocircuit $\{b,c,d,e\}$ imply that
$(X\cup T\cup d,Y-T)$ is $3$\dash separation of $M$.
Since $|Y-T|\geq 4$, it follows that $M$ has
a \ftv, which is impossible.
\end{subproof}

\begin{subsublemma}
\label{YminusT}
$|Y-T|\leq 5$.
\end{subsublemma}

\begin{subproof}
By~\ref{XdcupT} and \ref{einYd}, we see that
$(X\cup T,Y-(T\cup e))$ is a $3$\dash separation in
$M/d\ba e$.
As $\{b,c,d\}$ is a triad in $M\ba e$, it follows that
$d\in\cl_{M\ba e}^{*}(T)$, so
\[(X\cup T\cup d,Y-(T\cup e))\]
is a $3$\dash separation of $M\ba e$.
Since $|X\cup T\cup d|>4$, and $M\ba e$ is $(4,4)$\dash connected,
it follows that $|Y-(T\cup e)|\leq 4$, so $|Y-T|\leq 5$.
\end{subproof}

\begin{subsublemma}
\label{Ydnot5fan}
$|Y-T|=4$.
\end{subsublemma}

\begin{subproof}
We have observed that $|Y-T|\geq 4$, so if \ref{Ydnot5fan}
is false, it follows from \ref{YminusT} that $|Y-T|=5$.
From~\ref{XdcupT} and the dual of Proposition~\ref{zhou2.15}, we see that
$Y-T$ is a $5$\dash fan of $M/d$.
Let $(y_{1},\ldots, y_{5})$ be a fan ordering of
$Y-T$.
Since $M\ba e$ is $3$\dash connected, $e$ is contained in no
triads of $M$, so $e=y_{1}$ or $e=y_{5}$.
By reversing the fan ordering as necessary, we can assume that
the first case holds.
As $\{y_{2},y_{3},y_{4}\}$ is a triad of $M$, it follows
that $\{y_{3},y_{4},y_{5}\}$ is not a triangle, or else
$M$ has a $4$\dash fan.
Therefore $\{y_{3},y_{4},y_{5},d\}$ is a circuit of $M$
that is contained in $(Y-T)\cup d$.
It meets the cocircuit $\{b,c,d,e\}$ in a single element,
violating orthogonality.
\end{subproof}

As $|Y|\geq 5$, and $|Y-T|=4$, it follows that
$|Y|=5$ and $|Y\cap T|=1$.
From Proposition~\ref{zhou2.15}, we see that
$Y$ is a $5$\dash fan of $M/d$.
Let $(y_{1},\ldots,y_{5})$ be a fan ordering of
$Y$ in $M/d$.
As $M\ba e$ is $3$\dash connected, $e$ is in no triad in
$M$, and hence in $M/d$, so $e=y_{1}$ or $e=y_{5}$.
By reversing the fan ordering as necessary, we assume
$e=y_{1}$.
Since $\{y_{2},y_{3},y_{4}\}$ is a triad of $M$, it follows that
$\{y_{3},y_{4},y_{5}\}$ is not a triangle, or else
$M$ has a $4$\dash fan.
Therefore $\{y_{3},y_{4},y_{5},d\}$ is a circuit of $M$.
This circuit cannot meet the cocircuit $\{b,c,d,e\}$ in
the single element $d$.
Therefore the single element in $T\cap Y$
is in $\{y_{3},y_{4},y_{5}\}$.
Call this element $y$.
As the triangle $T$ cannot meet the triad
$\{y_{2},y_{3},y_{4}\}$ in a single element, it
follows that $y=y_{5}$.
Since $(y_{5},y_{4},y_{3},y_{2})$ is a $4$\dash fan
of $M/d$, and $\{y_{2},y_{3},y_{4}\}$ is a triad of
$M$, it follows from Lemma~\ref{trianglespersist} and
Lemma~\ref{minorsof45fans} that $M/d\ba y_{5}$,
and hence $M\ba y_{5}$ has an $N$\dash minor.
This contradicts the fact that $y_{5}$ is in the
triangle $T$.
Thus we have completed the proof of \ref{3sepsMcond}.
\end{subproof}

From \ref{3sepsMcond} we know that $M/d$ is $(4,4)$\dash connected.
Next we must eliminate the possibility that $M/d$ has a
$4$\dash fan.

\begin{sublemma}
\label{Ydisa4fan}
Let $(X,Y)$ be a \ftv\ of $M/d$, where
$|X\cap\{a,b,c\}|\geq 2$.
Then $Y$ is not a $4$\dash fan of $M/d$.
\end{sublemma}

\begin{subproof}
Assume that $Y$ is a $4$\dash fan,
$(y_{1},y_{2},y_{3},y_{4})$.
Thus $\{y_{2},y_{3},y_{4}\}$ is a triad
in $M/d$, and hence in $M$.
It follows from \ref{3sepsMcond} that $e\in Y$.
But $e$ is not in a triad of $M$, so $e=y_{1}$.
Since $M$ has no $4$\dash fan, it follows that
$\{e,y_{2},y_{3}\}$ is not a triangle of $M$, so
$\{e,d,y_{2},y_{3}\}$ is a circuit.

From \ref{Mcond3con}, we see that $M/d\backslash e$ is
$3$\dash connected.
We shall show that it is \ifc.
Once we prove this assertion, we will have shown that
$(M,N)$ is not a counterexample to Theorem~\ref{main}, since
$M/d\ba e$ has an $N$\dash minor by \ref{Mdelecond}.
This contradiction will complete the proof
of \ref{Ydisa4fan}.

\begin{subsublemma}
\label{bcsplitup}
If $(U,V)$ is a \ftv\ of $M/d\ba e$, then
$\{b,c\}\nsubseteq U$ and $\{b,c\}\nsubseteq V$.
\end{subsublemma}

\begin{subproof}
Assume that $(U,V)$ is a \ftv\ of $M/d\ba e$
such that $b,c\in U$.
Then $d\in\cl_{M\ba e}^{*}(U)$, because of the
triad $\{b,c,d\}$, so $(U\cup d,V)$ is a
\ftv\ in $M\ba e$.
As $|U\cup d|>4$, and $M\ba e$ is $(4,4)$\dash connected,
we deduce that $|V|=4$.
Assume that $V$ is a quad of $M\ba e$.
Then $V\cup e$ is a cocircuit of $M$, which cannot meet the
circuit $\{e,d,y_{2},y_{3}\}$ in a single
element.
Hence $y_{2}$ or $y_{3}$ is in $V$.
However, $V$ is a circuit in $M\ba e$, and $\{y_{2},y_{3},y_{4}\}$ is a
cocircuit in $M\ba e$, as it is a triad of $M$, and $M\ba e$
is $3$\dash connected.
Orthogonality requires that
$|V\cap \{y_{2},y_{3},y_{4}\}|=2$.
This means that $\{y_{2},y_{3},y_{4}\}\subseteq \cl_{M\ba e}^{*}(V)$,
so $V\cup\{y_{2},y_{3},y_{4}\}$ is a $5$\dash element
$3$\dash separating set in $M\ba e$.
As $M\ba e$ is $(4,4)$\dash connected, it follows that $M\ba e$
has at most nine elements, contradicting our earlier
assumption that $|E(M)|\geq 11$.
Thus $V$ is not a quad of $M\backslash e$, and
Proposition~\ref{quad4fan} implies that $V$ is a $4$\dash fan in $M\ba e$.

Let $T^{*}$ be the triad of $M\ba e$ that is contained in $V$.
As $M$ has no $4$\dash fans, $T^{*}\cup e$ is a cocircuit of
$M$.
It cannot meet the circuit $\{e,d,y_{2},y_{3}\}$ in the single
element $e$.
Let $y$ be an element in $\{y_{2},y_{3}\}\cap T^{*}$.
As $\{y_{2},y_{3},y_{4}\}$ is a triad in $M/d$, and hence
in $M$, it does not contain any element that is in a triangle
of $M$, or else $M$ has a $4$\dash fan.
Therefore $y$ is not in the triangle of $M\ba e$ that is contained in
$V$, so $V-y$ is a triangle of $M$.
Thus $V-y$ is contained in the ground set of every $N$\dash minor of
$M$, so Lemma~\ref{minorsof45fans} implies that
$M\ba e/y$, and hence $M/y$ has an $N$\dash minor.
However, since $y$ is contained in the triad $\{y_{2},y_{3},y_{4}\}$
of $M$, this contradicts Lemma~\ref{trianglespersist}.
\end{subproof}

Let $(U,V)$ be a \ftv\ of $M/d\ba e$, and assume that $a\in U$.
By~\ref{bcsplitup}, we may assume that
$x \in U$ and $y \in V$, where $\{x,y\} = \{b,c\}$.
Then $y\in \cl_{M/d\ba e}(U)$, as $\{a,x,y\}$
is a triangle, so $(U\cup y,V-y)$ is a $3$\dash separation
of $M/d\ba e$.
It follows from~\ref{bcsplitup} that
it is not a \ftv, so $|V|=4$.
Since $V$ contains an element that is in $\cl_{M/d\ba e}(U)$,
it cannot be a quad of $M/d\ba e$, so it is a $4$\dash fan.
Moreover, as $y \in \cl_{M/d\ba e}(U)$, it follows that
$y$ is not in the triad of $M/d\ba e$ that
is contained in $V$.
Therefore $V-y$ is a triad of $M/d\ba e$.
If $V-y$ is a triad of $M$, then $V-y$ is
contained in every $N$\dash minor of $M$.
Because $\{a,x,y\}=\{a,b,c\}$ is a triangle, it follows that
$y$ is contained in every $N$\dash minor.
Thus $V$ is in every $N$\dash minor of $M$.
This implies that $N$ has a $4$\dash element $3$\dash separating
set, which is impossible as $|E(N)|\geq 8$.
Therefore $V-y$ is not a triad of $M$, so
$(V-y)\cup e$ is a cocircuit.
It cannot meet the circuit $\{e,d,y_{2},y_{3}\}$ in the single
element $e$, so either $y_{2}$ or $y_{3}$ is in the
triad $V-y$ of $M\ba e$.

Note that $V-y$ is not equal to $\{y_{2},y_{3},y_{4}\}$,
as one set is a triad of $M$ and the other is not.
They are both triads of $M\ba e$, and they have at least
one element in common.
Hence they have exactly one element
in common, as $M\ba e$ is $3$\dash connected.
Let $z$ be the unique element in $(V-y)\cap\{y_{2},y_{3}\}$.
If $T'$ is the triangle of $M/d\ba e$ that is contained
in $V$, then $z$ is not in $T'$, as otherwise the triad
$\{y_{2},y_{3},y_{4}\}$ in $M/d$ meets the triangle $T'$
in a single element, $z$.
Therefore $V-z$ is a triangle in $M/d\ba e$, and
$V-y$ is a triad.

Note that $y$ is in every $N$\dash minor of $M$, because of the
triangle $\{a,x,y\}=\{a,b,c\}$, and $z$ is in every $N$\dash minor
of $M$ because of the triad $\{y_{2},y_{3},y_{4}\}$.
But $V$ cannot be contained in an $N$\dash minor
of $M/d\ba e$, as it is a $3$\dash separating set.
Therefore there is some element $w\in V-\{y,z\}$ such that
$N$ is a minor of $M/d\ba e\ba w$ or
$M/d\ba e/w$.
But $z$ is in the $2$\dash cocircuit $V-\{y,w\}$ of the first
matroid, and $y$ is in the $2$\dash circuit $V-\{z,w\}$
of the second.
This leads to a contradiction, as $y$ and $z$ are in the
ground set of every $N$\dash minor of $M$.

We conclude that there can be no \ftv\ in $M/d\ba e$, and therefore
$M/d\ba e$ is \ifc\ and has an $N$\dash minor.
This contradicts our assumption that $M$ is a counterexample
to Theorem~\ref{main}.
Thus \ref{Ydisa4fan} holds.
\end{subproof}

Now we can complete the proof of Lemma~\ref{fanlemma}.
If $(a,b,c,d)$ is a $4$\dash fan of $M\ba e$, where this
matroid is $(4,4)$\dash connected with an $N$\dash minor, then
\ref{Mdelecond} implies that $M/d$ has an $N$\dash minor.
It follows from \ref{Mcond3con} that $M\ba e/d$ and $M/d$
are $3$\dash connected, and \ref{3sepsMcond} implies that
$M/d$ is $(4,4)$\dash connected.
Moreover, if $(X,Y)$ is a \ftv\ of $M/d$ where
$X$ contains at least two elements of $\{a,b,c\}$, then
\ref{3sepsMcond} also implies that $|Y|=4$ and
$Y\cap\{a,b,c,e\}=\{e\}$.
As \ref{Ydisa4fan} implies that $Y$ cannot be a $4$\dash fan in
$M/d$, Proposition~\ref{quad4fan} implies that $Y$ is a quad.
Thus Lemma~\ref{fanlemma} is proved.
\end{proof}

\begin{lemma}
\label{quad44con}
Let $(M,N)$ be either $(\bar{M},\bar{N})$ or $(\bar{M}^{*},\bar{N}^{*})$.
Assume that the element $e$ is such that $M\ba e$
is $(4,4)$\dash connected with an $N$\dash minor, and that
$Q$ is a quad of $M\ba e$.
If $x\in Q$ and $M\ba e\ba x$ has an $N$\dash minor,
then $M\ba x$ is $(4,4)$\dash connected.
In particular, if $(X,Y)$ is a \ftv\ of $M\ba x$ such that
$|X\cap(Q-x)|\geq 2$, then $Y$ is a quad of $M\ba x$ such that
$|Y\cap Q|=1$, and $e\in Y$.
\end{lemma}

\begin{proof}
As $M$ has no quads, we deduce that $Q\cup e$ is a cocircuit in
$M$.

\begin{sublemma}
\label{Mdelx3con}
$M\ba x\ba e$ and $M\ba x$ are $3$\dash connected.
\end{sublemma}

\begin{subproof}
Let $(U,V)$ be a
$2$\dash separation in $M\ba x\ba e$.
By relabeling as necessary, we assume that $|U\cap (Q-x)|\geq 2$.
If $U$ contains $Q-x$, then $(U\cup x,V)$ is a $2$\dash separation
of $M\ba e$.
This is impossible, so $V$ contains a single element of
$Q-x$.
Then
\[\lambda_{M\ba x\ba e}(V-(Q-x))\leq
\lambda_{M\ba x\ba e}(V)\leq 1,\]
as $Q-x$ is a triad of $M\ba x\ba e$.
Now $x\in\cl_{M\ba e}(U\cup(Q-x))$, so
\[\lambda_{M\ba e}(V-(Q-x))=
\lambda_{M\ba x\ba e}(V-(Q-x))\leq 1.\]
But $M\ba e$ is $3$\dash connected, so this means that
$|V-(Q-x)|\leq 1$.
Thus $|V|=2$, and $V$ must be a $2$\dash cocircuit of $M\ba x\ba e$.
This means that $x$ is in a triad of $M\ba e$.
This triad must meet $Q$ in two elements, by orthogonality.
Thus $|\cl_{M\ba e}^{*}(Q)|\geq 5$, and $M\ba e$
contains a $5$\dash element $3$\dash separating set.
This is a contradiction as $M\ba e$ is $(4,4)$\dash connected
with at least ten elements.
Thus $M\ba x \ba e$ is $3$\dash connected, and it follows easily
that $M\ba x$ is $3$\dash connected.
\end{subproof}

Let $(X,Y)$ be a \ftv\ of $M\ba x$, and assume that
$X$ contains at least two elements of $Q-x$.
If $Q-x\subseteq X$, then $x\in\cl_{M}(X)$, as $Q$ is a circuit of $M$.
This implies that $(X\cup x,Y)$ is a \ftv\ of $M$, which
is impossible.
Therefore $Y$ contains exactly one element of $Q-x$.
Let us call this element $y$.

\begin{sublemma}
\label{New}
$(X-e, Y-e)$ is a $3$\dash separation of
$M\ba x \ba e$ and $y \in \cl_{M\ba x \ba e}^{*}(Y - \{e,y\})$.
 \end{sublemma}
 
\begin{subproof}
The fact that $(X- e, Y-e)$ is a $3$\dash separation of
$M\ba x\ba e$ follows because $|X|,|Y|\geq 4$, and
$M\ba x\ba e$ is $3$\dash connected.
Since $Q - x$ is a triad of
$M\ba x \ba e$, and $Q-\{x,y\}\subseteq X - e$, we deduce that
$y \in \cl_{M\ba x\ba e}^{*}(X-e)$.
This means that $y \in \cl_{M\ba x \ba e}^{*}(Y - \{e,y\})$,
as otherwise $((X-e)\cup y, Y-\{e,y\})$ is a
$2$\dash separation of $M\ba x\ba e$.
\end{subproof}

\begin{sublemma}
\label{lambdaYminusey}
$\lambda_{M\ba e}(Y-\{e,y\})\leq 2$.
\end{sublemma}

\begin{subproof}
Since $(X-e,Y-e)$ is a $3$\dash separation of $M\ba x\ba e$
and $y\in \cl_{M\ba x\ba e}^{*}(X-e)$, it follows that
\[
\lambda_{M\ba x\ba e}(Y-\{e,y\})\leq
\lambda_{M\ba x\ba e}(Y-e)=2.
\]
Since $(X-e)\cup \{x,y\}$ contains $Q$, it follows that
$x\in\cl_{M\ba e}((X-e)\cup y)$.
This means that
\[
\lambda_{M\ba e}(Y-\{e,y\})=
\lambda_{M\ba x\ba e}(Y-\{e,y\})\leq 2,
\]
as desired.
\end{subproof}

\begin{sublemma}
\label{Yatmost6}
$|Y|\leq 6$.
\end{sublemma}

\begin{subproof}
As $M\ba e$ is $(4,4)$-connected, $|Y - \{e,y\}| \le 4$ by
\ref{lambdaYminusey}.
Thus $|Y| \le 6$.
\end{subproof}

\begin{sublemma}
\label{Ynot6}
$|Y|\ne 6$.
\end{sublemma}

\begin{subproof}
Assume that $|Y|=6$.
If $e\notin Y$, then \ref{lambdaYminusey} implies that
$((X-e)\cup \{x,y\},Y-y)$ is a $3$\dash separation of $M\ba e$.
As $|Y-y|=5$ and $|(X-e)\cup\{x,y\}|=|X|+1\geq 5$, this contradicts the
fact that $M\ba e$ is $(4,4)$\dash connected.
Therefore $e\in Y$, and $Y-\{e,y\}$ is a $4$\dash element
$3$\dash separating set in $M\ba e$.
Thus $Y-\{e,y\}$ is either a quad or a $4$\dash fan of $M\ba e$.
The next two assertions show that both these cases are
impossible, thereby finishing the proof of \ref{Ynot6}.

\begin{subsublemma}
\label{Yminuseynotquad}
$Y-\{e,y\}$ is not a quad of $M\ba e$.
\end{subsublemma}

\begin{subproof}
Assume that $Y-\{e,y\}$ is a quad of $M\ba e$.
Thus it is a circuit of $M$, and $Y-y$ is a cocircuit of $M$.
If $Y-y$ is not a cocircuit of $M\ba x$, then
$x$ is in the coclosure of $Y-y$ in $M$.
This leads to a contradiction to orthogonality with the
circuit $Q$ of $M$.
Thus $Y-y$ is a cocircuit of $M\ba x$, so
$Y-\{e,y\}$ is a quad in both $M\ba e$ and $M\ba x\ba e$.
As $Y-y$ is a cocircuit of $M\ba x$ and $|Y|=6$, we see that
$r_{M\ba x}^{*}(Y)\geq 4$, so
\[
r_{M\ba x}(Y)\leq
\lambda_{M\ba x}(Y)-r_{M\ba x}^{*}(Y)+|Y|
\leq 2-4+6=4.
\]

By \ref{New}, there is a cocircuit of $M\ba x\ba e$
contained in $Y-e$ that contains $y$.
The symmetric difference of this cocircuit with
$Y-\{e,y\}$ is a disjoint union of cocircuits.
As $M\ba x\ba e$ contains no cocircuit with fewer than three
elements, it follows that there are two triads,
$T_{1}^{*}$ and $T_{2}^{*}$, of $M\ba x\ba e$, such that
$T_{1}^{*}\cap T_{2}^{*}=\{y\}$, and
$T_{1}^{*}\cup T_{2}^{*}=Y-e$.
If both $T_{1}^{*}\cup e$ and $T_{2}^{*}\cup e$ are cocircuits
of $M\ba x$, then we can take the symmetric
difference of these cocircuits, and deduce that $Y-\{e,y\}$
is a cocircuit of $M\ba x$ that is properly contained in the
cocircuit $Y-y$.
Since this is impossible, we deduce that we can relabel as
necessary, and assume that $T_{1}^{*}$ is a triad of $M\ba x$.

Let $z$ be an arbitrary element of $T_{2}^{*}-y$.
Then $Y-\{e,y,z\}$ is independent in $M\ba x$.
Orthogonality with the cocircuit $(Q-x)\cup e$ means that
$y$ cannot be in the closure of $Y-\{e,y,z\}$ in $M\ba x$.
Thus $Y-\{e,z\}$ is independent.
Since $r_{M\ba x}(Y)\leq 4$, it follows that
$Y-\{e,z\}$ spans $Y$ in $M\ba x$.
Let $C$ be a circuit of $M\ba x$ such that
$\{e\}\subseteq C\subseteq Y-z$.
If $y\notin C$, then $C$ and the cocircuit $(Q-x)\cup e$ meet
in $\{e\}$.
Therefore $y\in C$.
This implies that $C$ contains exactly one element of $T_{1}^{*}-y$.
Now $C$ cannot be a triangle, as $M\ba e$ has an $N$\dash minor.
Therefore $C$ also contains the single element in $T_{2}^{*}-\{y,z\}$.
But now the circuit $C$ meets the cocircuit $Y-y$ of $M\ba x$ in
three elements: $e$, and a single element from each of $T_{1}^{*}-y$
and $T_{2}^{*}-y$.
This contradiction proves~\ref{Yminuseynotquad}.
\end{subproof}

\begin{subsublemma}
\label{Yminuseynot4fan}
$Y-\{e,y\}$ is not a $4$\dash fan of $M\ba e$.
\end{subsublemma}

\begin{subproof}
Assume that $Y-\{e,y\}=(y_{1},y_{2},y_{3},y_{4})$ is a
$4$\dash fan in $M\ba e$.
Then $\{y_{2},y_{3},y_{4},e\}$ is a cocircuit of $M$.

As $M\ba x\ba e$ has no cocircuits with fewer than three
elements, $\{y_{2},y_{3},y_{4}\}$ is a triad of $M\ba x\ba e$.
Thus $(y_{1},y_{2},y_{3},y_{4})$ is a $4$\dash fan of
$M\ba x\ba e$.
By~\ref{New}, there is a cocircuit $C^*$ of $M\ba x\ba e$
such that $\{y\}\subseteq C^{*}\subseteq Y-e$.
This cocircuit must meet the triangle
$\{y_{1},y_{2},y_{3}\}$ in exactly two elements.
If $\{y_{2},y_{3}\}\subseteq C^{*}$, then
the symmetric difference of $\{y_{2},y_{3},y_{4}\}$ and $C^{*}$
is $\{y,y_{4}\}$, as $\{y_{2},y_{3},y_{4}\}$ is not properly
contained in $C^{*}$.
Since $M\ba x\ba e$ has no $2$\dash cocircuit, 
we deduce that $y_{1}\in C^{*}$.
Either $C^{*}$, or its symmetric difference with
$\{y_{2},y_{3},y_{4}\}$, is a triad of $M\ba x\ba e$ that
contains $y$, $y_{1}$, and a single element from $\{y_{2},y_{3}\}$.
We can swap the labels on $y_{2}$ and $y_{3}$ if necessary,
so we can assume that $\{y,y_{1},y_{2}\}$ is a triad.
Thus $(y,y_{1},y_{2},y_{3},y_{4})$ is a $5$\dash cofan of $M\ba x\ba e$.
The dual of Lemma~\ref{minorsof45fans} implies that
$M\ba x\ba e/y$ and $M\ba x\ba e/y_{4}$ have $N$\dash minors.

Recall that $\{y_{2},y_{3},y_{4},e\}$ is a
cocircuit of $M$.
It is also a cocircuit of $M\ba x$, as otherwise
$x\in \cl_{M}^{*}(\{y_{2},y_{3},y_{4},e\})$, and this
contradicts orthogonality with the circuit $Q$.
Therefore $\{y_{2},y_{3},y_{4}\}$ is coindependent in
$M\ba x$.

Assume that $y_{1}$ is in $\cl_{M\ba x}^{*}(\{y_{2},y_{3},y_{4}\})$.
Let $C^{*}$ be a cocircuit of $M\ba x$ such that
$\{y_{1}\}\subseteq C^{*}\subseteq \{y_{1},y_{2},y_{3},y_{4}\}$.
As $\{y_{1},y_{2},y_{3}\}$ is a circuit of $M\ba x$, orthogonality
requires that $C^{*}$ contains exactly one element from
$\{y_{2},y_{3}\}$.
Thus
$C^{*}$ is a triad of $M\ba x$.
It is also a triad of $M$, or else $C^{*}\cup x$ is a cocircuit
of $M$ that violates orthogonality with $Q$.
Thus $M$ has a triad that contains two elements from the triangle
$\{y_{1},y_{2},y_{3}\}$, so $M$ has a $4$\dash fan.
This contradiction implies that
$y_{1}\notin\cl_{M\ba x}^{*}(\{y_{2},y_{3},y_{4}\})$.
Thus $\{y_{1},y_{2},y_{3},y_{4}\}$ is a coindependent set
in $M\ba x$, so $r_{M\ba x}^{*}(Y)\geq 4$.
Now we see that
\[
r_{M\ba x}(Y)=
\lambda_{M\ba x}(Y)-r_{M\ba x}^{*}(Y)+|Y|
\leq 2-4+6=4.
\]

If $\{y,y_{1},y_{3},y_{4}\}$ is dependent in $M\ba x\ba e$,
then it is a circuit, by orthogonality with the
triads $\{y,y_{1},y_{2}\}$ and $\{y_{2},y_{3},y_{4}\}$, and the
fact that $M\ba x\ba e$ has no $2$\dash circuits.
In this case, $\{y,y_{1},y_{3},y_{4}\}$ is a circuit of
$M$, and $Q\cup e$ is a cocircuit that meets it in the single
element $y$.
Therefore $\{y,y_{1},y_{3},y_{4}\}$ is independent in $M\ba x\ba e$,
and hence in $M\ba x$.
Therefore
$\{y,y_{1},y_{3},y_{4}\}$ spans $Y$ in $M\ba x$.
Let $C$ be a circuit of $M\ba x$ such that
$\{e\}\subseteq C\subseteq \{e,y,y_{1},y_{3},y_{4}\}$.

First observe that $y \in C$, as otherwise $C$ and $Q \cup e$
are a circuit and a cocircuit of $M$ that meet in $\{e\}$.
We have noted that $\{e,y_{2},y_{3},y_{4}\}$ is a cocircuit
of $M\ba x$.
Therefore orthogonality implies that
$C$ contains exactly one element of $\{y_{3},y_{4}\}$.
Since $M\ba e$ has an $N$\dash minor, it follows that
$e$ is in no triangles of $M$.
Therefore $y_{1}$ must be in $C$.
Hence $C$ is either $\{e,y,y_{1},y_{3}\}$ or
$\{e,y,y_{1},y_{4}\}$.
In the first case, we take the symmetric difference of
$C$ with the triangle $\{y_{1},y_{2},y_{3}\}$, and discover that
$e$ is in the triangle $\{e,y,y_{2}\}$, a contradiction.
Therefore $C=\{e,y,y_{1},y_{4}\}$.

We noted earlier that $M\ba x\ba e/y$, and hence
$M/y$, has an $N$\dash minor.
The symmetric difference of $C$ with the triangle
$\{y_{1},y_{2},y_{3}\}$ is $\{e,y,y_{2},y_{3},y_{4}\}$, which
must therefore be a circuit of $M$.
Thus $\{e,y_{2},y_{3},y_{4}\}$ is a circuit of $M/y$.
It is also a cocircuit, as it is a cocircuit in $M$.
Thus $M/y$ contains a quad that contains $e$.
Since $M/y\ba e$ has an $N$\dash minor, it follows
from Lemma~\ref{quadiso} that if $z$ is an arbitrary
member of the quad $\{e,y_{2},y_{3},y_{4}\}$, then
$M/y\ba z$ has an $N$\dash minor.
In particular, $M/y\ba y_{2}$, and hence $M\ba y_{2}$
has an $N$\dash minor.
This is contradictory, as $y_{2}$ is contained in the
triangle $\{y_{1},y_{2},y_{3}\}$ of $M$.
This completes the proof of~\ref{Yminuseynot4fan}.
\end{subproof}

The proof of \ref{Ynot6} now follows immediately from
\ref{Yminuseynotquad} and \ref{Yminuseynot4fan}.
\end{subproof}

\begin{sublemma}
\label{Ynot5}
$|Y|\ne 5$.
\end{sublemma}

\begin{subproof}
Assume that $|Y|=5$.
If $e\in X$, then $y\in \cl_{M\ba x}^{*}(X)$, since
$(Q-x)\cup e$ is a cocircuit of $M\ba x$ that is contained in
$X\cup y$.
This means that $(X\cup y, Y-y)$ is a $3$\dash separation of
$M\ba x$.
As $Q$ is a circuit of $M$, and $Q-x\subseteq X\cup y$, it
follows that $x\in \cl_{M}(X\cup y)$.
Therefore $(X\cup\{x,y\},Y-y)$ is a $3$\dash separation of
$M$, and as $|X\cup\{x,y\}|,|Y-y|\geq 4$, we have
violated the internal $4$\dash connectivity of $M$.
Therefore $e$ is in $Y$.

Proposition~\ref{zhou2.15} implies that $Y$ is a
$5$\dash cofan of $M\ba x$.
Let $(y_{1},y_{2},y_{3},y_{4},y_{5})$ be a fan ordering of
$Y$ in $M\ba x$.
Since $e$ is contained in no triangle of $M$ by
Lemma~\ref{trianglespersist}, we can assume that $e=y_{1}$.
The element $y$ cannot be contained in
$\{y_{2},y_{3},y_{4}\}$, or else this triangle meets the
cocircuit $Q\cup e$ of $M$ in the single element $y$.
Now $\{y_{1},y_{2},y_{3}\}$ is a triad of $M$, or
else $\{x,y_{1},y_{2},y_{3}\}$ is a cocircuit, and it
meets the circuit $Q$ in the single element $x$.
However, this is a contradiction, as $e=y_{1}$ and $M\ba e$
is $3$\dash connected.
\end{subproof}

We can now complete the proof of Lemma~\ref{quad44con}.
Recall that $(X,Y)$ is a \ftv\ of $M\ba x$, where
$x$ is contained in the quad $Q$ of $M\ba e$, and
$|X\cap (Q-x)|\geq 2$.
By combining \ref{Yatmost6}, \ref{Ynot6}, and
\ref{Ynot5}, we deduce that $|Y|=4$.
Therefore $M\ba x$ is $(4,4)$\dash connected with
an $N$\dash minor, and $Y$ is either a quad or a
$4$\dash fan of $M\ba x$.

Assume that $e$ is not in $Y$.
If $Y$ is a quad of $M\ba x$, then it is a circuit of $M$ that
meets the cocircuit $Q\cup e$ in the single element $y$.
Therefore $Y$ must be a $4$\dash fan of $M\ba x$.
Certainly $y$ is not contained in the triangle of $Y$, by
orthogonality with $Q\cup e$.
Therefore $Y=(y_{1},y_{2},y_{3},y)$ is a $4$\dash fan.
We can apply Lemma~\ref{fanlemma} to $M\ba x$, and deduce that
$M/y$ is $(4,4)$\dash connected with an $N$\dash minor.
Since $M/y$ is not \ifc, Lemma~\ref{fanlemma} also
implies that $M/y$ contains a quad and that this quad
contains $x$.
However, $Q-y$ is a triangle in $M/y$, so $M/y$ contains
a quad, and a triangle that contains an element of this quad.
It follows that the triangle and the quad meet in two
elements, and their union is a $5$\dash element
$3$\dash separating set of $M/y$.
As $M/y$ is $(4,4)$\dash connected, this leads to a contradiction,
so now we know that $e$ is in $Y$.

If $Y$ is a $4$\dash fan of $M\ba x$, then $e$ is not
contained in the triangle of this fan, as $M\ba e$ has an
$N$\dash minor.
Therefore $y$ is contained in a triangle of $M\ba x$ that
is contained in $Y-e$.
This triangle violates orthogonality with the cocircuit
$Q\cup e$ in $M$.
Hence $Y$ is a quad of $M\ba x$, and Lemma~\ref{quad44con} holds.
\end{proof}

The next proof is essentially the same as an argument used in
\cite{CMO12}.

\begin{lemma}
\label{nodeletequads}
Let $(M,N)$ be either $(\bar{M},\bar{N})$ or $(\bar{M}^{*},\bar{N}^{*})$.
Assume that the element $e$ is such that $M\ba e$
is $(4,4)$\dash connected with an $N$\dash minor, and that
$Q$ is a quad of $M\ba e$.
Then $N$ is not a minor of $M\ba e\ba x$, for any
element $x\in Q$.
\end{lemma}

\begin{proof}
Assume that $N$ is a minor of $M\ba e \ba x$ for some element
$x$ of $Q$.
By Lemma~\ref{quadiso}, deleting any
element of $Q$ from $M\ba e$ produces a matroid with an
$N$\dash minor.
Lemma~\ref{quad44con} implies that $M\ba x$ is $(4,4)$\dash connected
and contains a quad, $Q_{x}$, such that $e\in Q_{x}$, and
$|Q\cap Q_{x}|=1$.

\begin{sublemma}
\label{dualquads}
Assume that $x_{1}$ and $x_{2}$ are elements of
$Q$, and that $M\ba x_{1}$ contains a quad, $Q_{1}$,
such that $e\in Q_{1}$, and
$Q\cap Q_{1}=\{x_{2}\}$.
Then $M\ba x_{2}$ contains a quad, $Q_{2}$, such that
$Q\cap Q_{2}=\{x_{1}\}$ and $Q_{1}\cap Q_{2}=\{e\}$.
\end{sublemma}

\begin{subproof}
Since $M\ba e\ba x_{2}$ has an $N$\dash minor,
we can apply Lemma~\ref{quad44con}, and deduce that
$M\ba x_{2}$ contains a quad $Q_{2}$ such that
$e\in Q_{2}$ and $|Q\cap Q_{2}|=1$.
On the other hand, $M\ba x_{1}$ is $(4,4)$\dash connected,
and contains a quad, $Q_{1}$.
Moreover, $M\ba x_{1}\ba x_{2}$ is isomorphic to
$M\ba x_{1}\ba e$, by
Lemma~\ref{quadiso} and the fact that $e$ and
$x_{2}$ are both in $Q_{1}$, so $M\ba x_{1}\ba x_{2}$ has an $N$\dash minor.
Hence we can apply Lemma~\ref{quad44con} again, and
deduce that $M\ba x_{2}$ contains a quad $Q_{2}'$
such that $x_{1}\in Q_{2}'$ and
$|Q_{1}\cap Q_{2}'|=1$.

We will show that $Q_{2}=Q_{2}'$.
Assume this is not the case.
As $Q_{2}$ and $Q_{2}'$ are both quads
of $M\ba x_{2}$, orthogonality demands that
they are disjoint, or they meet in two elements.
In the latter case,
$Q_{2}\symdif Q_{2}'$ is a circuit of $M$,
and
$(Q_{2}\cup x_{2})\symdif (Q_{2}'\cup x_{2})=
Q_{2}\symdif Q_{2}'$
must be a cocircuit of $M$,
so $M$ has a quad.
As this is impossible, we deduce that $Q_{2}$
and $Q_{2}'$ are disjoint.
Therefore $e\notin Q_{2}'$, as $e$ is in $Q_{2}$.
This means that $|Q\cap Q_{2}'|=2$,
as otherwise the circuit $Q_{2}'$ and the
cocircuit $Q\cup e$ meet in $\{x_{1}\}$.
But $Q_{2}'\cup x_{2}$ is a cocircuit, and
$Q$ is a circuit, and they meet in three elements:
$x_{2}$ and the two elements of
$Q\cap Q_{2}'$.
This contradiction shows that $Q_{2}=Q_{2}'$, so
$x_{1}\in Q_{2}$.
Furthermore, $Q\cap Q_{2}=\{x_{1}\}$ and
$Q_{1}\cap Q_{2}=\{e\}$.
\end{subproof}

Now we return to the proof of Lemma~\ref{nodeletequads}.
Let $y$ be the single element in $Q\cap Q_{x}$.
By \ref{dualquads}, we see that $M\ba y$ has a
quad $Q_{y}$ such that
$Q\cap Q_{y}=\{x\}$ and $Q_{x}\cap Q_{y}=\{e\}$.

Let $\{z,w\}=Q-\{x,y\}$.
We can again apply Lemma~\ref{quad44con} and deduce the existence
of $Q_{z}$, a quad of $M\ba z$ that contains $e$ and
a single element of $Q$.
Note that $Q_{z}\ne Q_{y}$, or else we can take the
symmetric difference of $Q_{y}\cup y$ and
$Q_{z}\cup z$ and deduce that $\{y,z\}$ is a series
pair of $M$.
Assume that $y\in Q_{z}$.
As the cocircuit $Q_{z}\cup z$ and the circuit
$Q_{y}$ both contain $e$, and $x$ is not the
single element in $Q\cap Q_{z}$,
it follows that one of the elements in $Q_{y}-\{x,e\}$ is in $Q_{z}$.
Then the cocircuit $Q_{y}\cup y$ and the circuit
$Q_{z}$ meet in three elements: $e$, $y$, and an
element in $Q_{y}-\{x,e\}$.
Therefore the single element in $Q\cap Q_{z}$
is neither $y$ nor $z$, so it is $x$ or $w$.

First we assume that $Q\cap Q_{z}=\{w\}$.
Then \ref{dualquads} implies that $M\ba w$
has a quad $Q_{w}$ such that
$Q\cap Q_{w}=\{z\}$ and $Q_{z}\cap Q_{w}=\{e\}$.
The cocircuit $Q_{w}\cup w$ and the circuit $Q_{x}$
both contain the element $e$.
Moreover, $y\notin Q_{w}$, so 
there is an element $\alpha$ in $(Q_{x}-\{e,y\})\cap Q_{w}$.
Let $\beta$ be the unique element in
$Q_{x}-\{e,y,\alpha\}$.
Similarly, the cocircuit $Q_{w}\cup w$ and the
circuit $Q_{y}$ have $e$ in common, but
$x\notin Q_{w}$, so there is an element
$\gamma$ in $(Q_{y}-\{e,x\})\cap Q_{w}$.
Thus $Q_{w}=\{e,z,\alpha,\gamma\}$.
Consider the set
$X=\{x,y,z,\alpha,\beta\}$.
It spans: $w$ because of the circuit $Q$;
$e$ because of the circuit $Q_{x}$;
$\gamma$ because it spans the circuit
$Q_{w}$; and $Q_{y}$ because it spans
$x$, $e$, and $\gamma$.
This shows that $Q\cup Q_{x}\cup Q_{y}$ is a
$9$\dash element set satisfying
$r(Q\cup Q_{x}\cup Q_{y})\leq 5$.
Moreover, $X$ cospans $e$ because of the cocircuit
$Q_{x}\cup x$.
It cospans $w$ because it cospans $e$, and
$Q\cup e$ is a cocircuit.
Now it cospans: $\gamma$ as $Q_{w}\cup w$ is a cocircuit;
and $Q_{y}$ as $Q_{y}\cup y$ is a cocircuit.
Thus $X$ spans and cospans $Q\cup Q_{x}\cup Q_{y}$, so
\[
\lambda_{M}(Q\cup Q_{x}\cup Q_{y})
\leq 5+5-9=1.
\]
As $M$ is $3$\dash connected, this means that there is at
most $1$ element not in $Q\cup Q_{x}\cup Q_{y}$.
This is a contradiction as $|E(M)|\geq 11$.
Hence we conclude that $Q\cap Q_{z}=\{x\}$.

Now the cocircuit $Q_{z}\cup z$ and the circuit
$Q_{x}$ both contain $e$, so
$|Q_{z}\cap (Q_{x}-\{e,y\})|=1$.
But this means that the circuit
$Q_{z}$ and the cocircuit $Q_{x}\cup x$ have three
elements in common: $e$, $x$, and the element
in $Q_{z}\cap (Q_{x}-\{e,y\})$.
This contradiction completes the proof of
Lemma~\ref{nodeletequads}.
\end{proof}

\begin{lemma}
\label{no4fans}
Let $(M,N)$ be either $(\bar{M},\bar{N})$ or $(\bar{M}^{*},\bar{N}^{*})$.
Assume that the element $e$ is such that $M\ba e$
is $(4,4)$\dash connected with an $N$\dash minor.
Then $M\ba e$ has no $4$\dash fans.
\end{lemma}

\begin{proof}
Assume that $(a,b,c,d)$
is a $4$\dash fan of $M\ba e$.
It follows from Lemma~\ref{fanlemma} that $M/d$ is
$(4,4)$\dash connected with an $N$\dash minor.
Since it is not \ifc, it contains a quad $Q$ such that
$Q\cap \{a,b,c,e\}=\{e\}$.
We will show that $M\ba e/d$ is \ifc, and this
will contradict the fact that $M$ and $N$ provide a counterexample
to Theorem~\ref{main}, thereby proving Lemma~\ref{no4fans}.
Note that Lemma~\ref{fanlemma} states that $M\ba e/d$ is
$3$\dash connected with an $N$\dash minor.

\begin{sublemma}
\label{QminusenotinX}
Let $(X',Y')$ be a \ftv\ of $M\ba e/d$.
Then neither $X'$ nor $Y'$ contains $Q-e$.
\end{sublemma}

\begin{subproof}
Assume that $Q-e\subseteq X'$.
As $Q$ is a circuit of $M/d$, this means that
$e\in \cl_{M/d}(X')$.
Thus $(X'\cup e, Y')$ is a \ftv\ of $M/d$.
Lemma~\ref{fanlemma} says that one side of this \ftv\ is a
quad that contains $e$.
But this is impossible as $|X'\cup e|>4$, and $e\notin Y'$.
\end{subproof}

Let $(X,Y)$ be a \ftv\ of $M\ba e /d$, and assume that
$|X\cap (Q-e)|\geq 2$.
By \ref{QminusenotinX} we see that there is a
single element in $Y\cap (Q-e)$.
Let this element be $y$.
Since $Q-e$ is a triad in $M\ba e/d$, it follows that
$y\in \cl_{M\ba e/d}^{*}(X)$.
Therefore $(X\cup y, Y-y)$ is a
$3$\dash separation in $M\ba e/d$, but
\ref{QminusenotinX} implies that it is not a \ftv.
Hence $|Y|=4$.
Orthogonality with the triad $Q-e$ implies that $Y$ is not a
quad of $M\ba e/d$.
Thus $Y=\{y_{1},y_{2},y_{3},y_{4}\}$, where
$(y_{1},y_{2},y_{3},y_{4})$ is a $4$\dash fan in
$M\ba e /d$.
Orthogonality also implies that $y=y_{4}$.

Assume that $M\ba e/d/y$ has an $N$\dash minor.
Then $M^{*}\ba d\ba y$ has an $N^{*}$\dash minor.
As $M^{*}\ba d$ is $(4,4)$\dash connected, and $y$ is in the quad
$Q$ of this matroid, we now have a contradiction to
Lemma~\ref{nodeletequads}.
Therefore $M\ba e/d/y$ has no $N$\dash minor.
Lemma~\ref{minorsof45fans} implies that
$M\ba e/d\ba y_{1}$, and hence $M\ba y_{1}$, has an $N$\dash minor.
From this, we deduce that $\{y_{1},y_{2},y_{3}\}$ is not a
triangle of $M$, so $\{d,y_{1},y_{2},y_{3}\}$ is a circuit.
Since $\{b,c,d,e\}$ is a cocircuit, this implies that
exactly one $b$ or $c$ is in $\{y_{1},y_{2},y_{3}\}$.
Let $\alpha$ be the single element in $\{b,c\}\cap\{y_{1},y_{2},y_{3}\}$.
Then $\alpha\ne y_{1}$, as $M\ba e/d\ba y_{1}$ has an
$N$\dash minor, and $b$ and $c$ are contained in a triangle of $M$.

Both $\{y_{2},y_{3},y\}$ and $\{a,b,c\}$ contain the
element $\alpha$.
As $\{y_{2},y_{3},y\}$ is a triad in $M\ba e/d$, and hence in
$M\ba e$, and $\{a,b,c\}$ is a triangle of $M\ba e$,
it follows that $\{y_{2},y_{3}\}=\{\alpha,a\}$,
since $y \in Q$ and $Q \cap \{a,b,c\} =\emptyset$.
Hence either $(y,a,b,c,d)$ or $(y,a,c,b,d)$ is a $5$\dash cofan of
$M\ba e$, depending on whether $\alpha=b$ or $\alpha=c$.
In either case, from Proposition~\ref{fans3sep},
and the fact that $M\ba e$ is $(4,4)$\dash connected,
we deduce that $|E(M\ba e)|\leq 9$, a contradiction.
Thus $M\ba e/d$ has no \ftv, and is therefore
\ifc.
This contradiction completes the proof of
Lemma~\ref{no4fans}.
\end{proof}

By Lemma~\ref{candeletee}, we know
we can choose $(M,N)$ to be $(\bar{M},\bar{N})$ or
$(\bar{M}^{*},\bar{N}^{*})$ in such a way that
$M\ba e$ is $(4,4)$\dash connected with an $N$\dash minor
for some element $e\in E(M)$.
From Lemma~\ref{no4fans}, we deduce that $M\ba e$
has no $4$\dash fans, and therefore contains at least one
quad.
Moreover, deleting any element from this quad destroys all
$N$\dash minors, by Lemma~\ref{nodeletequads}.
Therefore we next consider contracting an element from a
quad in $M\ba e$.

\begin{lemma}
\label{conquad}
Let $(M,N)$ be either $(\bar{M},\bar{N})$ or $(\bar{M}^{*},\bar{N}^{*})$.
Assume that the element $e$ is such that $M\ba e$
is $(4,4)$\dash connected with an $N$\dash minor, and that
$Q$ is a quad of $M\ba e$.
If $x\in Q$, then $M\ba e/x$ is $3$\dash connected, and
$M/x$ is $(4,4)$\dash connected with an $N$\dash minor.
In particular, if $(X,Y)$ is a \ftv\ of $M/x$ such that
$|X\cap (Q-x)| \geq 2$, then $Y$ is a quad of $M/x$, and
$Y\cap (Q\cup e)=\{e\}$.
\end{lemma}

\begin{proof}
To see that $M/x$ has an $N$\dash minor, we note that $Q$ is not
contained in the ground set of any $N$\dash minor of $M\ba e$.
By Lemma~\ref{nodeletequads}, we cannot delete any element of $Q$
in $M\ba e$ and preserve an $N$\dash minor.
Therefore we must contract an element of $Q$.
By the dual of Lemma~\ref{quadiso}, we can contract any element.
Thus $M\ba e/x$, and hence $M/x$, has an $N$\dash minor.

\begin{sublemma}
\label{Mconx3con}
$M\ba e/x$ and $M/x$ are $3$\dash connected.
\end{sublemma}

\begin{subproof}
Assume that $(U,V)$ is a $2$\dash separation of $M\ba e/x$
such that $|U\cap (Q-x)|\geq 2$.
If $Q-x\subseteq U$, then $(U\cup x,V)$ is a $2$\dash separation
in $M\ba e$, as $Q$ is a cocircuit in this matroid.
Since $M\ba e$ is $3$\dash connected, this is not true, so
$V$ contains a single element, $y$, of $Q-x$.
Then $y\in \cl_{M\ba e/x}(U)$.
However $(U\cup y, V-y)$ is not a $2$\dash separation of
$M\ba e/x$, or else $(U\cup \{x,y\},V-y)$ is a
$2$\dash separation of $M\ba e$.
Thus $V$ is either a $2$\dash circuit or a $2$\dash cocircuit
in $M\ba e/x$.
Orthogonality with $Q-x$ tells us that the latter case is
impossible.
Therefore $x$ is in a triangle in $M\ba e$ that contains
two elements of $Q$.
The union of $Q$ with this triangle is a $5$\dash element
$3$\dash separating set in $M\ba e$,
contradicting the fact that $M\ba e$ is $(4,4)$\dash connected.
Therefore $M\ba e/x$ is $3$\dash connected.
If $M/x$ is not, then $e$ must be in a triangle with $x$ in $M$,
and this is impossible by Lemma~\ref{trianglespersist}.
\end{subproof}

Let $(X,Y)$ be a \ftv\ of $M/x$, and assume that
$|X\cap (Q-x)|\geq 2$.

\begin{sublemma}
\label{einY}
$e\in Y$.
\end{sublemma}

\begin{subproof}
Assume that $e\in X$.
If $Q-x\subseteq X$, then
$(X\cup x,Y)$ is a \ftv\ of $M$, which is impossible.
Therefore $Y\cap (Q-x)$ contains a single element, $y$.
Now $y\in \cl_{M/x}(X)$, so $(X\cup y, Y-y)$ is a
$3$\dash separation in $M/x$.
Therefore $|Y|=4$, or else
$(X\cup\{x,y\},Y-y)$ is a \ftv\ of $M$.
Orthogonality with $Q-x$ shows that $Y$ is not a quad
of $M/x$.
Therefore $Y$ is a $4$\dash element fan.
This means that $M^{*}\ba x$ is $(4,4)$\dash connected with an
$N^{*}$\dash minor and a $4$\dash fan, contradicting
Lemma~\ref{no4fans}.
\end{subproof}

\begin{sublemma}
\label{lambdaYminus}
$\lambda_{M\ba e}(Y-(Q\cup e))\leq 2$.
\end{sublemma}

\begin{subproof}
As $\lambda_{M/x}(Y)=2$, and $Q-x\subseteq \cl_{M/x}(X)$,
it follows that $\lambda_{M/x}(Y-Q)\leq 2$.
Therefore $\lambda_{M\ba e/x}(Y-(Q\cup e))\leq 2$.
Now $x$ is in the coclosure of the complement of
$Y-(Q\cup e)$ in $M\ba e$, as $Q$ is a cocircuit, so
\[
\lambda_{M\ba e}(Y-(Q\cup e))=\lambda_{M\ba e/x}(Y-(Q\cup e))\leq 2,
\]
as desired.
\end{subproof}

\begin{sublemma}
\label{Yatmostsix}
$|Y|\leq 6$.
\end{sublemma}

\begin{subproof}
Since $M\ba e$ is $(4,4)$\dash connected,
\ref{lambdaYminus} implies that
$|Y-(Q\cup e)| \leq 4$.
The result follows.
\end{subproof}

\begin{sublemma}
\label{Ynotsix}
$|Y|\ne 6$.
\end{sublemma}

\begin{subproof}
Assume that $|Y|=6$.
If $Q-x\subseteq X$, then \ref{lambdaYminus} implies that
$M\ba e$ has a $5$\dash element $3$\dash separating set,
which leads to a contradiction.
Therefore $Y\cap (Q-x)$ contains a single element, $y$.
The same argument implies that $e\in Y$.
Since $Q-x$ is a triangle in $M/x$, it follows that
$y\in \cl_{M/x}(X)$, so $(X\cup x,Y-y)$ is a
$3$\dash separation of $M/x$.
Proposition~\ref{zhou2.15} implies that $Y-y$ is a
$5$\dash fan of $M/x$.
Let $(y_{1},\ldots, y_{5})$ be a fan ordering of $Y-y$.
As $e$ is contained in no triads of $M$, we can assume that
$e=y_{1}$.
As $\{y_{2},y_{3},y_{4}\}$ is a triad of $M/x$, and hence
of $M$, it cannot be the case that $\{y_{3},y_{4},y_{5}\}$
is a triangle, or else $M$ has a $4$\dash fan.
Therefore $\{x,y_{3},y_{4},y_{5}\}$ is a circuit of
$M$ that meets the cocircuit $Q\cup e$ in the single
element $x$.
This contradiction completes the proof of \ref{Ynotsix}.
\end{subproof}

\begin{sublemma}
\label{Ynotfive}
$|Y|\ne 5$.
\end{sublemma}

\begin{subproof}
Assume that $|Y|=5$.
First suppose that $Q-x\subseteq X$.
Then $(X\cup x,Y-e)$ is a $3$\dash separation of $M\ba e$,
by \ref{lambdaYminus}.
Thus $Y-e$ is a quad of $M\ba e$, by Proposition~\ref{quad4fan}
and Lemma~\ref{no4fans}.
But Proposition~\ref{zhou2.15} implies that $Y$ is a
$5$\dash fan of $M/x$.
Thus $Y$ contains a triad of $M/x$, and hence of $M$.
This means that $Y-e$ contains a cocircuit of size at
most three in $M\ba e$, contradicting the fact that it
is a quad.
Thus $Y\cap (Q-x)$ contains a single element, $y$.

By again using Proposition~\ref{zhou2.15}, we see that
$Y$ is a $5$\dash fan of $M/x$.
Let $(y_{1},\ldots, y_{5})$ be a fan ordering.
Orthogonality with $Q-x$ means that $y$ is not contained in a
triad of $M/x$ that is contained in $Y$.
Therefore we can assume that $y=y_{1}$.
As $e$ is in no triad of $M$, it follows that
$e\notin \{y_{2},y_{3},y_{4}\}$.
As $M/x\ba e$ is $3$\dash connected, by \ref{Mconx3con},
we deduce that $(y_{1},y_{2},y_{3},y_{4})$
is a $4$\dash fan of $M/x\ba e$.
Lemmas~\ref{minorsof45fans} and~\ref{trianglespersist}
imply that $M/x\ba e\ba y_{1}$, and hence
$M\ba e\ba y_{1}$, has an $N$\dash minor.
As $y_{1}=y$ is contained in the quad $Q$ of $M\ba e$,
this means we have a contradiction to
Lemma~\ref{nodeletequads}.
\end{subproof}

Now we complete the proof of Lemma~\ref{conquad}.
By combining \ref{Yatmostsix}, \ref{Ynotsix}, and
\ref{Ynotfive}, we see that $M/x$ is $(4,4)$\dash connected.
Since it is not \ifc, $Y$ has four elements, and is therefore
a quad of $M/x$, by an application of Lemma~\ref{no4fans}.
We know that $e\in Y$ by \ref{einY}, but orthogonality with
the triangle $Q-x$ implies that $Y$ contains no other element
of $Q\cup e$.
\end{proof}

Finally, we are in a position to prove Theorem~\ref{main}.
By Lemma~\ref{candeletee}, we can assume that $(M,N)$ is
either $(\bar{M},\bar{N})$ or $(\bar{M}^{*},\bar{N}^{*})$,
and $M\ba e$ is $(4,4)$\dash connected with an $N$\dash minor,
for some element $e$.
Lemma~\ref{no4fans} implies that $M\ba e$ has no $4$\dash fans.
As it is not \ifc, it contains a quad $Q$.
Deleting any element of $Q$ destroys all $N$\dash minors,
by Lemma~\ref{nodeletequads}, so $M\ba e/x$ has an $N$\dash minor,
for some element $x \in Q$.
Lemma~\ref{conquad} says that $M\ba e/x$ is $3$\dash connected,
and $M/x$ is $(4,4)$\dash connected.
As $M/x$ is not \ifc, it has a quad, $Q_{x}$, such that
$(Q\cup e)\cap Q_{x}=\{e\}$.
We will show that $M\ba e/x$ is \ifc, and this will provide
a contradiction that completes the proof of
Theorem~\ref{main}.

Assume that $(X,Y)$ is a \ftv\ of $M\ba e /x$,
where $|X\cap (Q_{x}-e)|\geq 2$.
If $Q_{x}-e\subseteq X$, then $(X\cup e,Y)$ is a \ftv\ of $M/ x$,
as $Q_{x}$ is a circuit in $M/x$.
Then Lemma~\ref{conquad} implies that either $X\cup e$
or $Y$ is a quad that contains $e$, an impossibility.
Therefore $Y\cap (Q_{x}-e)$ contains a single element, $y$.
In $M\ba e/x$, the set $Q_{x}-e$ is a triad, so
$y$ is in the coclosure of $X$.
Therefore $(X\cup y, Y-y)$ is a $3$\dash separation.
If it is a \ftv, then $(X\cup\{y,e\},Y-y)$ is 
a \ftv\ of $M/x$, and this leads to the same contradiction as
before.
Therefore $|Y|=4$.
Orthogonality with $Q_{x}-e$ shows that $Y$ is not a quad of
$M\ba e/x$.
Thus we assume that $(y_{1},y_{2},y_{3},y_{4})$ is a $4$\dash fan
and a fan ordering of $Y$ in $M\ba e/x$.
Then $y=y_{4}$, by orthogonality with $Q_{x}-e$ and
$\{y_{1},y_{2},y_{3}\}$.

Since $y$ is in a quad of $M/x$, Lemma~\ref{nodeletequads}
implies that $M/x/y$, and hence $M\ba e/x/y$ has no
$N$\dash minor.
Therefore Lemma~\ref{minorsof45fans} implies that
$M\ba e/x\ba y_{1}$ has an $N$\dash minor.
This shows that $\{y_{1},y_{2},y_{3}\}$ is not a triangle of
$M$, so $\{x,y_{1},y_{2},y_{3}\}$ is a circuit.
As $Q\cup e$ is a cocircuit, there is a single element,
which we call $z$, in $(Q-x)\cap \{y_{1},y_{2},y_{3}\}$.

Note that $\{y_{2},y_{3},y\}$ is not a triad of $M$,
by orthogonality with the circuit $Q_{x}\cup x$.
Therefore $\{y_{2},y_{3},y,e\}$ is a cocircuit.
This means that $z$ is not in $\{y_{2},y_{3}\}$,
for otherwise $\{y_{2},y_{3},y,e\}$ meets the circuit $Q$
in the single element $z$.
Therefore $z=y_{1}$, and $M\ba e/x\ba z$ has an $N$\dash minor.
This means that $M\ba e\ba z$ has an $N$\dash minor, and as
$z$ is in $Q$, we have contradicted Lemma~\ref{nodeletequads}.
Thus Theorem~\ref{main} is now proved.


\end{document}